\documentclass[11pt]{amsart}

\usepackage{amssymb,amsthm}
\usepackage{graphicx}
\usepackage{mathptmx}  
\usepackage[utf8]{inputenc}
\usepackage{amssymb}
\usepackage{enumitem}
\setenumerate[0]{label=(\roman*)} 
\usepackage{stmaryrd}
\usepackage[english]{babel}
\usepackage{amsmath, latexsym}
\usepackage{tikz}
\usetikzlibrary{decorations.markings}
\usepackage{subcaption}

\usepackage[round]{natbib}

\usepackage{tikz-cd}
\tikzset{commutative diagrams/diagrams={transform shape, 
  nodes={scale=1.00}} }

\usepackage[colorlinks=true,allcolors=blue!50!black]{hyperref}
\usepackage[all]{hypcap}
\usepackage{relsize}

\usepackage{lmodern}

\newtheorem{theorem}{Theorem}[section]
\newtheorem{lemma}[theorem]{Lemma}

\newtheorem{proposition}[theorem]{Proposition}

\theoremstyle{definition}
\newtheorem{definition}[theorem]{Definition}

\newtheorem{example}[theorem]{Example}

\newcommand{\NN}{\mathbb{N}}
\newcommand{\RR}{\mathbb{R}}
\newcommand{\ZZ}{\mathbb{Z}}
\newcommand{\FF}{\mathbb{F}}

\newcommand{\cl}{\operatorname{cl}}
\newcommand{\id}{\operatorname{id}}
\newcommand{\Inv}{\operatorname{Inv}}
\newcommand{\Int}{\operatorname{int}}

\newcommand{\Tel}{\operatorname{Tel}}
\newcommand{\Tor}{\operatorname{T}}
\newcommand{\Torp}{\operatorname{T_\bullet}}
\newcommand{\iT}{\#}  
\renewcommand{\phi}{\varphi}

\newcommand{\mapP}{\kappa}
\newcommand{\mapQ}{\lambda}

\newcommand{\Ho} {\operatorname{H}}
\newcommand{\rHo} {\widetilde{\operatorname{H}}}
\newcommand{\CT}{\operatorname{C}\Tor}
\newcommand{\CTp}{\operatorname{C}\Torp}
\newcommand{\incl}{j}
\numberwithin{equation}{section}
\newcommand{\suspend}{I_f}

\DeclareRobustCommand{\coprod}{\mathop{\text{\fakecoprod}}}
\newcommand{\fakecoprod}{%
  \sbox0{$\prod$}%
  \smash{\raisebox{\dimexpr.9625\depth-\dp0}{\scalebox{1}[-1]{$\prod$}}}%
  \vphantom{$\prod$}%
}

\begin{document}

\title[Conley index and mapping torus]
{The Conley index for discrete dynamical systems and the mapping torus}

\author{Frank Weilandt}
\address{
  Institute for Algebra, Geometry, Topology
  and their Applications (ALTA),
  Department of Mathematics,
  University of Bremen,
  Bibliothekstr. 1,
  28359 Bremen,
  Germany
}
\email{weilandt@uni-bremen.de}

\date{January 19, 2018}

\begin{abstract}
The Conley index for flows is a topological invariant
describing the behavior around an isolated invariant set $S$.
It is defined as the homotopy type of a quotient space
$N/L$, where $(N,L)$ is an index pair for $S$.
In the case of a discrete dynamical system,
i.e., a continuous self-map $f\colon X\to X$,
the definition is similar. But one needs to
consider the index map $f_{(N,L)}\colon  N/L\to N/L$ induced by $f$.
The Conley index in this situation is defined as
the homotopy class $[f_{(N,L)}]$ modulo shift equivalence.
The shift equivalence relation is rarely used outside this context
and not well understood in general.
For practical purposes like numerical computations, one needs to use weaker
algebraic invariants for distinguishing Conley indices,
usually homology. Here we consider a topological invariant:
the homotopy type of the mapping torus of the index map $f_{(N,L)}$.
Using a homotopy type offers new ways for comparing
Conley indices -- theoretically and numerically.
We present some basic properties and examples,
compare it to the definition via shift equivalence and sketch an idea for its construction
using rigorous numerics.
\end{abstract}

\maketitle

\section{Introduction}
\label{intro}
The classical Conley index for flows as introduced by \citet{C}
describes the local behavior of 
a flow around an isolated invariant set $S$. It is defined
using a so-called index pair $L\subset N$ of compact sets
in the phase space: The set
$S$ is the invariant part of $\cl(N\setminus L)$ and $S\subset\Int(N\setminus L)$.
A trajectory that leaves $N$ has to pass through $L$.
The Conley index of $S$ is then
the pointed homotopy type of the quotient $N/L$.

Starting with \citet{RS}, an analogue of this index
for discrete dynamical systems (continuous self-maps) has been developed.
The definition by \citet{FR},
equivalent to the one presented by \citet{Sz},
uses the homotopy class of a so-called index map 
$[f_{(N,L)}\colon N/L\to N/L]$ up to shift equivalence.
We call this quite general version of the Conley index
the \emph{shift equivalence index} in this article.
Even though this index is homotopy theoretic in spirit,
it is not defined as the homotopy
type of a topological space -- in contrast to the version for flows.
Algebraic versions were already introduced before the shift equivalence index:
One can use functors like homology and then equivalence
relations on linear maps (\citealt{M90}).

In this article, we consider the following invariant of the
shift equivalence index:
the homotopy type of the mapping torus of the index map $f_{(N,L)}$.
The main features of this \emph{mapping torus index} are:
Since it is the homotopy type of a space, it offers additional ways of extracting
information than the homological Conlex index usually used for
numerical computations. A shift equivalence
between maps induces a homotopy equivalence between their mapping tori.
The reduced mapping
torus index is the classical flow Conley index for the suspension semiflow.
The fundamental group of the mapping torus index
contains information not apparent in the homological Conley index.
And its construction as a cell complex seems to be feasible using
rigorous numerics as described in \citet{KMM} -- at least 
as a cell complex with the correct homology.

The definition of the mapping torus index and why it makes sense is presented in 
Sect.~\ref{sec:defIndex} and Sect.~\ref{sec:defnWellDefined}. Its main properties
are shown in Sect.~\ref{sec:MainProperties}. In Sect.~\ref{sec:shiftEqui},
we compare our definition with the shift equivalence index.
Under strong assumptions, a homotopy equivalence of mapping tori
can yield a shift equivalence of self-maps.
After recalling algebraic invariants of the mapping torus, we consider some examples
in Sect.~\ref{sec:examples}.
In Sect.~\ref{sec:compareFloer}, we show that
the reduced mapping torus index coincides with the flow Conley index
of the suspension semiflow.
Sect.~\ref{sec:numerics} sketches how mapping tori could be constructed
without having full information about a map, but only certain enclosures
of its graph.

Despite being apparently a coarse invariant, the mapping torus offers
new ways for understanding shift equivalence. Representing the Conley index
as a space, one can potentially use methods and algorithms developed
for comparing homotopy types.

\paragraph{Basic definitions}
We use the following basic notions of pointed and unpointed topology.
For a topological space $X$ and a subspace $Y \subset X$ with $Y\neq \varnothing$,
we consider the quotient space $X/Y$ as a pointed space with base point
$[Y]$. We let $X/\varnothing := X \coprod \{*\}$, the disjoint union
of $X$ and the one-point space $\{*\}$. An asterisk $*$ 
usually denotes the respective base point.

Given a pointed space $(X,x_0)$ and an unpointed space $Y$, they form
a reduced product:
\[X \rtimes Y := (X,x_0) \rtimes Y := \frac{X \times Y}{\{x_0\} \times Y}.\]
A homotopy from $f$ to $g$ is a continuous map
$H \colon X \times [0,1] \to Y$
such that $H(\cdot,0) = f$ and $H(\cdot,1) = g$.
Given pointed (base point preserving) maps on pointed spaces,
a pointed homotopy is a pointed continuous map
$H \colon (X,x_0) \rtimes [0,1] \to (Y,y_0)$ with analogous properties.
We often omit the base point when this does not lead to confusion.

\section{Definition of the mapping torus index}
\label{sec:defIndex}

From here throughout this article, we let $X$ be a locally compact metric space,
and we let $f\colon X \to X$ be a 
discrete dynamical system, i.e., a continuous map.
Let $M \subset X$ and $x \in M$.
A \emph{solution of $f$ in $M$ through $x$} is a sequence
$\gamma \colon \ZZ \to M$ such that for all $n \in \ZZ: \gamma(n+1) = f(\gamma(n))$
and $\gamma(0) = x$.
The \emph{invariant subset of $M$} is
\[ \Inv(M,f) = \{ x \in M \mid \text{ there is a solution of $f$ in $M$ through $x$}\}. \]
We call a set $S\subset X$ \emph{isolated invariant}
if it is compact and has a neighborhood $M$ such that $\Inv(M,f)=S$.

We use the definition of index pairs from \citet{RS}:
If $A\subset B \subset X$, where $A$ and $B$ are compact,
we call $(B,A)$ a \emph{compact pair}.
Given an isolated invariant set $S\subset X$,
a compact pair $(N,L)$ is an \emph{index pair for $(S,f)$} if
\[ \Inv(\cl(N\setminus L),f) = S \subset \Int(N\setminus L)\]
and the map
$f_{(N,L)}\colon N/L\to N/L$,
\[
f_{(N,L)}([x]):=\begin{cases} [f(x)] &\text{if $x,f(x)\in N\setminus L$,}
                \\
                [L] &\text{otherwise,}
                \end{cases}
\]
is continuous. In this case, we call $f_{(N,L)}$ the \emph{index map}.

For a continuous map $\mapP\colon P \to P$ on some space $P$, we let
its \emph{(unreduced) mapping torus} be
\[ \Tor(\mapP) := \frac{P \times [0,1]}{(x,1) \sim (\mapP(x),0)}. \]
For a pointed continuous map $\mapP\colon (P,p_0) \to (P,p_0)$ on some pointed space $(P,p_0)$, let
its \emph{reduced mapping torus} be
\[ \Torp(\mapP) := \frac{(P,p_0) \rtimes [0,1]}{(x,1) \sim (\mapP(x),0)}. \]
Its homotopy type depends only on the homotopy class of $\mapP$
(see \citealt{R1}, Prop. 6.1(i)).
This gives us two ways of defining a mapping torus (Conley) index.

\begin{definition}
\label{defn:mappingTorusIndex}
Let $(N,L)$ be an index pair for $(S,f)$.
The \emph{(unreduced) mapping torus index} of $(S,f)$
is the homotopy type of $\Tor(f_{(N,L)})$ (an unpointed space). We write
\[ \CT(S,f) := [\Tor(f_{(N,L)})]. \]
The \emph{reduced mapping torus index} of $(S,f)$ is the pointed homotopy type of
$\Torp(f_{(N,L)})$,
we abbreviate this as 
\[\CTp(S,f) := [\Torp(f_{(N,L)})].\]
\end{definition}

We start with basic examples (more are given in Sect.~\ref{sec:examples}). 
The empty invariant set $S=\varnothing$ has an index pair $(\varnothing,\varnothing)$
with index map the pointed map on the one-point space $\{*\}$.
We call the mapping torus index of $(\varnothing,f)$ \emph{trivial}.
For the definitions above, this means:
The unreduced mapping torus index $\CT(S,f)$ is trivial
if it is the homotopy type of the circle $S^1$.
The reduced mapping torus index $\CTp(S,f)$ is trivial if it is the pointed homotopy
type of the one-point space $\{*\}$.

Even though the index map is pointed, we mainly consider
the unreduced mapping torus in this article.
It can contain finer information:
For example, if $f_{(N,L)}$ is the degree-$2$ map on the circle $S^1$, then
$\pi_1(\Torp(f_{(N,L)}))$ is the trivial group, whereas
$\pi_1(\Tor(f_{(N,L)})$ is not isomorphic to the fundamental
group of the circle $S^1$, as we show in Example~\ref{ex:degree2}.

\section{The mapping torus index is well-defined}
\label{sec:defnWellDefined}

Using that the shift equivalence index is well-defined (\citealt{FR}),
one could apply Theorem~\ref{thm:shiftEquiHtpyEqui} to show
that the mapping torus index is also well-defined. In this section,
we present a more direct proof. We first recall Theorem~\ref{thm:torus_choice_index_pair},
which was already shown by \citet{RS}.
But there it is assumed that $X$ is a manifold and $f$ a diffeomorphism.
We recall details of the theory therein to show that the theorem also
holds in our context of a self-map on the metric space $X$.
For a continuous map $\mapP \colon P \to P$, consider the map
\begin{align*}
(P \times [0,1])\times[0,\infty) &\to \Tor(\mapP), \\
 ((x,\theta),t) &\mapsto [\mapP^{\lfloor \theta+t\rfloor}(x), \theta+t-\lfloor \theta+t\rfloor],
\end{align*}
where $\lfloor\cdot\rfloor\colon [0,\infty)\to \NN$
denotes the floor function.
This map is continuous and sends $((x,1),t)$ and $((\mapP(x),0),t)$ to the same point
\[ (\mapP^{\lfloor 1+t\rfloor}(x), 1+t-\lfloor 1+t\rfloor)
= (\mapP^{\lfloor t\rfloor}(\mapP(x)), t - \lfloor t \rfloor ) \in P\times [0,1). \]
Hence, it induces the continuous \emph{suspension semiflow}
\begin{align*}
\phi_\mapP\colon \Tor(\mapP) \times [0,\infty)&\to \Tor(\mapP),\\
 ([x,\theta],t) &\mapsto [\mapP^{\lfloor \theta+t\rfloor}(x), \theta+t-\lfloor \theta+t\rfloor].
\end{align*}

Let $\incl_\mapP \colon P \to \Tor(f),\ \incl_\mapP(x) = [x,0].$
Given maps $\mapP\colon P \to P$, $\mapQ\colon Q\to Q$ and
a map $r\colon P \to Q$ such that $\mapQ r= r \mapP$, let
the induced map
$r_\iT\colon \Tor(\mapP) \to \Tor(\mapQ)$ be given by
$r_\iT[x,\theta] = [r(x),\theta]$.
This definition makes the following diagram commute.
\begin{equation}
\label{diagram:induced_map}
\begin{tikzcd}
  P \arrow{r}{\mapP} \arrow{d}{r} & P \ar{d}{r} \ar{r}{j_\mapP} & \Tor(\mapP) \ar[dashed]{d}{r_\iT}\\
  Q \ar{r}{\mapQ}                 & Q \ar{r}{j_\mapQ} & \Tor(\mapQ)
\end{tikzcd}
\end{equation}

\begin{lemma}
\label{lem:induce_id}
If $P=Q$, $\mapP=\mapQ$ and $r=\mapP^n$ for some $n\geq 0$,
then the induced map $r_\iT=\mapP^n_\iT$
is homotopic to the identity on $\Tor(\mapP).$
In particular, $j_\mapP \mapP^n \simeq j_\mapP$.
\end{lemma}

\begin{proof}
The suspension semiflow defines a homotopy because
\begin{align*}
\id_{\Tor(\mapP)}[x,\theta] &= [x,\theta] = \phi_\mapP([x,\theta],0) \text{ and} \\
\mapP_\iT^n[x,\theta] &= [\mapP^n(x),\theta] = \phi_\mapP([x,\theta],n).
\end{align*}
\end{proof}

We want to relate two indices given different index pairs
for $(S,f)$. We recall the proof of Theorem~\ref{thm:propGeneralIndexMaps},
which is basically Theorem~6.3 from \citet{RS}.
There it was originally stated for flows and
invertible discrete systems. We only need it for discrete systems here and
the invertibility assumption was not used in the original proof.
\begin{lemma}[\citealt{RS}, Theorem~4.3]
\label{prop:cont}
For a compact pair $(N,L)$, the map $f_{(N,L)}$ as defined above
is continuous if and only if both of the following conditions are fulfilled
for every $x_0 \in f^{-1}(N\setminus L)$:
\begin{enumerate}
\item\label{prop:cont1} If $x_0\in L$, then there 
is an open set $U\subset X$ such that $x_0\in U$ and
$f(U\cap N\setminus L)\subset X\setminus N$.
\item\label{prop:cont2} If $x_0\in N\setminus L$, then
there is an open set $U\subset X$ such that $x_0\in U$ and 
$f(U\cap N\setminus L) \subset N\setminus L$.
\end{enumerate}
\end{lemma}

Define for an arbitrary subset $M \subset X$ and $n\geq 0$:
\[ \Inv^n(M,f) = \{f^n(x) \mid x, f(x), \ldots, f^{2n}(x) \in M\}. \]
Let $(N_\alpha,L_\alpha)$ and $(N_\beta,L_\beta)$ be index pairs for $(S,f)$.
Then there is a number
$n\geq 0$ such that
\[ \Inv^n(N_\beta\setminus L_\beta,f) \subset N_\alpha\setminus L_\alpha \text{ and }
   \Inv^n(N_\alpha\setminus L_\alpha,f) \subset N_\beta\setminus L_\beta. \]
Let $u=u(\alpha,\beta)$ be the smallest $n\geq 0$ with this property.
Obviously, $u(\alpha,\beta) = u(\beta,\alpha)$, and we get
the following property right from the definition of $u=u(\alpha,\beta)$.
\begin{lemma}
\label{lemma:minusU}
For any $x\in X$,
if $f^{[0,2u]}(x) \subset N_\alpha\setminus L_\alpha$, then
$f^u(x) \in N_\beta\setminus L_\beta$.
\end{lemma}
Now we define
\[ C_{\alpha\beta}:=\{x\in N_\alpha\setminus L_\alpha \mid
  f^{[0,2u]}(x) \subset N_\alpha\setminus L_\alpha \text{ and }
    f^{[u+1,3u+1]}(x)\subset N_\beta\setminus L_\beta\}\]
and the (not necessarily continuous) map
\begin{align*}
  f_{\beta\alpha} \colon N_\alpha/L_\alpha & \to N_\beta/L_\beta, \\
  x & \mapsto \begin{cases} 
	f^{3u+1}(x) & \text{if $x\in C_{\alpha\beta}$,}\\
    [L_\beta] & \text{otherwise}.
  \end{cases}
\end{align*}
A special case is $\alpha=\beta$. Then $u(\alpha,\alpha)=0$ and
$f_\alpha:=f_{\alpha\alpha}=f_{(N_\alpha,L_\alpha)}$. The following theorem
allows us to compare index maps.

\begin{theorem}[\citealt{RS}, Theorem~6.3]
\label{thm:propGeneralIndexMaps}
\leavevmode
\begin{enumerate}
\item\label{continuity} $f_{\beta\alpha}$ is continuous,
\item\label{composition1} 
$f_{\alpha\beta}\circ f_{\beta\alpha} = f_\alpha^{6u(\alpha,\beta)+2}$,
\item\label{composition2} $f_{\beta\alpha}\circ f_\alpha = f_\beta \circ f_{\beta\alpha}$.
\end{enumerate}
\end{theorem}

\begin{proof}
The idea for the proof of~\ref{continuity} is to consider
five cases depending on where $x_0$ lies within $N_\alpha/L_\alpha$,
and to show for each case that $f_{\beta\alpha}$ is continuous in $x_0$.
We do not recall all cases here,
but only present one difficult case from the proof in \citealt{RS}
slightly adapted to our needs here.
The proofs of the other four cases are similar or shorter.
We mainly use Lemma~\ref{lemma:minusU} and the continuity of the index maps
$f_\alpha$ and $f_\beta$.

\emph{Case $x_0 \in C_{\alpha\beta}$.}
We mainly need to show that there is an open set $U\subset X$ such that $x_0 \in U$ and
$U\cap N_\alpha\setminus L_\alpha \subset C_{\alpha\beta}$.

Note that $f^i(x_0)\in f^{-1}(N_\alpha\setminus L_\alpha)$
for $i \in \{0,\ldots, 2u-1\}$ by definition of $C_{\alpha\beta}$.
Using Lemma~\ref{prop:cont}\ref{prop:cont2},
there are open sets $U_i\subset X$ such that $f^i(x_0) \in U_i$ and
\[ f(U_i \cap N_\alpha\setminus L_\alpha) \subset N_\alpha\setminus L_\alpha.\]
Since $x_0 \in C_{\alpha\beta}$, Lemma~\ref{lemma:minusU} yields $f^u(x_0) \in N_\beta\setminus L_\beta$.
Now, applying Lemma~\ref{prop:cont}\ref{prop:cont2} to the index pair
$(N_\beta,L_\beta)$ yields: For each $i\in\{u,\ldots,3u\}$, there is
an open set $V_i\subset X$ such that $f^i(x_0) \in V_i$ and
\[ f(V_i \cap N_\beta\setminus L_\beta) \subset N_\beta\setminus L_\beta.\]
For $0\leq i \leq 3u$, we define open sets
\[ W_i := \begin{cases}
  U_i &\text{ if $0\leq i \leq u-1$},\\
  U_i \cap V_i &\text{ if $u \leq i \leq 2u-1$},\\
  V_i &\text{ if $2u\leq i \leq 3u$}.
\end{cases}\]
Now we let 
\[U:=\bigcap_{i=0}^{3u} f^{-i}(W_i) \subset X,\]
an open set.
Let $x\in U\cap N_\alpha\setminus L_\alpha$.
For $0\leq i\leq 2u-1$, we have the implication
\[ f^i(x)\in U_i \cap N_\alpha\setminus L_\alpha
\implies f^{i+1}(x)\in U_{i+1}\cap N_\alpha\setminus L_\alpha.\]
Overall, $f^{[0,2u]}(x)\subset N_\alpha\setminus L_\alpha$,
and therefore, by
Lemma~\ref{lemma:minusU}, $f^u(x)\in U_u \cap N_\beta\setminus L_\beta$.
This implies $f^{[u,3u+1]}(x)\in N_\beta\setminus L_\beta$.
Hence, $x\in C_{\alpha\beta}$.

Therefore, $f_{\beta\alpha}(x)=f^{3u+1}(x)$ for every
$x\in U\cap N_\alpha\setminus L_\alpha$.
Note that $U\cap N_\alpha\setminus L_\alpha$
is open in $N_\alpha\setminus L_\alpha$
and therefore open in $N_\alpha/L_\alpha$.
Since $f$ is continuous,
$f_{\beta\alpha}$ is continuous in $x_0$.
This finishes the proof of the case $x_0 \in C_{\alpha\beta}$.

The statements~\ref{composition1} and~\ref{composition2} are special cases of
Theorem~6.3(iii) in \citet{RS}.
\end{proof}

\begin{theorem}
\label{thm:torus_choice_index_pair}
The mapping torus index of $(S,f)$ is independent of the choice of an
index pair $(N,L)$.
\end{theorem}

\begin{proof}
Let $(N_\alpha,L_\alpha)$ and $(N_\beta,L_\beta)$ be index pairs
for $(S,f)$.
Now let $r:=f_{\beta\alpha}$, $s:=f_{\alpha\beta}$ and $n:=6u(\alpha,\beta)+2$.
Theorem~\ref{thm:propGeneralIndexMaps} shows that
\begin{enumerate}
\item $r f_\alpha = f_\beta r$ and
  $s f_\beta = f_\alpha  s$,
\item $s r = f_\alpha^n$ and $r s = f_\beta^n$.
\end{enumerate}
Then, Lemma~\ref{lem:induce_id} yields a homotopy equivalence
\[ s_\iT r_\iT = (sr)_\iT = (f_\alpha^n)_\iT \simeq \id_{\Tor(f_\alpha)}, \]
and similarly for $r_\iT s_\iT$.
Therefore $\Tor(f_\alpha) \simeq \Tor(f_\beta)$.
\end{proof}

Similarly, $\CTp(S,f)$ is a well-defined pointed homotopy type.

\section{Main properties}
\label{sec:MainProperties}

It is possible to replace the index pair by a homotopy equivalent one
in the following sense.

\begin{proposition}
\label{prop:compareHtpyEqu}
If the map $r$ in Diagram~\eqref{diagram:induced_map} is a homotopy equivalence,
then so is $r_\iT$.
\end{proposition}

\begin{proof}
Diagram~\eqref{diagram:induced_map} is a special case of Diagram~
\eqref{eq:htpyCommutativeInduced} by putting $H(x,\theta)=[r(x),\theta]$
and then $r_\iT = (r,H)_\iT$.
Theorem~\ref{thm:shiftEquiHtpyEqui} yields the result.
\end{proof}

An important property of the usual Conley index definitions is the invariance
under continuation. 
Consider a collection $\{(S_t,f_t) \mid t \in [0,1]\}$
of sets $S_t \subset X$ and maps $f_t \colon X\to X$,
such that the dynamical system
\begin{align*}
F\colon X\times [0,1] &\to X\times [0,1],\\
(x,t) &\mapsto (f_t(x),t),
\end{align*}
is continuous and the set $\Sigma \subset X \times [0,1]$ given by
\[ \Sigma = \{(x,t) \mid x \in S_t \} \]
is an isolated invariant set for $F$. The collection $\{S_t, f_t\}$ is called
a \emph{continuation from $(S_0,f_0)$ to $(S_1,f_1)$}. Note that
$f_s \simeq f_t$ for all $s,t\in [0,1]$.

\begin{theorem}
If there is a a continuation from $(S_0,f_0)$ to $(S_1,f_1)$,
then $\CT(S_0,f_0)=\CT(S_1,f_1)$.
\end{theorem}

\begin{proof}
Let $\{(S_t,f_t) \mid t \in [0,1]\}$ be a continuation of isolated
invariant sets. Then, applying Corollary~5.5 in \citet{RS}, there are
open sets $I_1, \ldots, I_n$ covering the unit interval $[0,1]$
and pairs $(N_1,L_1),\ldots, (N_n,L_n)$ such that each $(N_i, L_i)$
is an index pair when $(S_t,f_t)$ for $t \in I_i$.
We assume $0\in I_1$ and $1\in I_n$.
Now one can observe:
\begin{enumerate}
\item If $s,t\in I_i$, then $\Tor(f_{s,(N_i,L_i)}) \simeq \Tor(f_{t,(N_i,L_i)})$
since $f_s\simeq f_t$.
\item If $t\in I_i\cap I_j$, then 
$(N_i,L_i)$ and $(N_j,L_j)$ are index pairs for $(S_t,f_t)$, hence
Theorem~\ref{thm:torus_choice_index_pair} yields
$\Tor(f_{t,(N_i,L_i)}) \simeq \Tor(f_{t,(N_j,L_j)})$.
\end{enumerate}
Since $[0,1]$ is connected and every $I_i$ is open in $[0,1]$, this shows that
\[ \CT(S_0,f_0)=[\Tor(f_{0,(N_1,L_1)})]=[\Tor(f_{1,(N_n,L_n)})] = \CT(S_1,f_1).\]
\end{proof}

The following result about compositions is an analogue of Theorem~1.12 from \citet{M94}.
\begin{theorem}[Commutativity]
Let $\phi \colon X \to Y$ and $\psi \colon Y \to X$ be continuous maps.
Consider the dynamical systems $f=\psi\phi$ and $g=\phi\psi$. Let
$S \subset X$ be an isolated invariant set for $f$.
Then $\phi(S)$ is an isolated invariant set for $g$ and
$\CT(S,f)=\CT(\phi(S),g)$.
\end{theorem}
\begin{proof}
Using the proof of Theorem~1.12 from \citet{M94}, index pairs $(N,L)$
for $(S,f)$ and $(M,K)$ for $(\phi(S),g)$ exist such that there
are continuous maps $\bar{\phi}\colon N/L\to M/K$
and $\bar{\psi}\colon M/K \to N/L$ with $\bar{\psi}\bar{\phi}=f_{(N,L)}$
and $\bar{\phi}\bar{\psi}=g_{(M,K)}$.
Hence, $\CT(S,f) = [\Tor(\bar\psi\bar\phi)] = [\Tor(\bar\phi\bar\psi)]
= \CT(\phi(S),g).$
The equality in the middle is a general property of
mapping tori (see \citealt{R1}, Prop. 6.1(ii)).
\end{proof}

\section{Definition via shift equivalence}
\label{sec:shiftEqui}

Homotopy classes of self-maps
$[\mapP\colon P \to P]$ and $[\mapQ \colon Q \to Q]$ are called \emph{shift equivalent}
if there are continuous maps $r\colon P \to Q$ and $s \colon Q \to P$ such that
$\mapQ r \simeq r \mapP$, $s \mapQ \simeq \mapP s$,
$s r \simeq \mapP^n$ and $r s \simeq \mapQ^n$ for some $n \in \NN$.
Here we call the shift equivalence class of $[f_{(N,L)}]$
the \emph{shift equivalence (Conley) index}. It was introduced
by \citet{FR}.
In this section, we show that the mapping torus index is strictly coarser,
but sometimes allows statements about shift equivalence if $N/L$ is
compact and connected.

Assume we are given maps $\mapP\colon P \to P$ and $\mapQ\colon Q\to Q$ and
a map $r\colon P \to Q$ such that $\mapQ r\simeq r\mapP$. This means
that $\incl_\mapQ r \mapP \simeq \incl_\mapQ \mapQ r \simeq \incl_\mapQ r$.
Hence, there is a homotopy
$H \colon P \times [0,1] \to \Tor(\mapQ)$ with $H(x,0) = \incl_\mapQ r(x)$
and $H(x,1) = \incl_\mapQ r\mapP(x)$ for
all $x \in P$.
Let the induced map
$(r,H)_\iT\colon \Tor(\mapP) \to \Tor(\mapQ)$ be given by
$(r,H)_\iT[x,\theta] = H(x,\theta)$.
This is well-defined because
$H(x,1) = [r\mapP(x),0] = H(\mapP(x),0).$
In the diagram
\begin{equation}
\label{eq:htpyCommutativeInduced}
\begin{tikzcd}
  P \arrow{r}{\mapP} \arrow{d}{r} & P \ar{d}{r} \ar{r}{\incl_\mapP} &
    \Tor(\mapP) \ar[dashed]{d}{(r,H)_\iT}\\
  Q \ar{r}{\mapQ}                 & Q \ar{r}{j_\mapQ} & \Tor(\mapQ),
\end{tikzcd}
\end{equation}
the left square is homotopy commutative and the right square is strictly commutative.
First we observe the following generalization of Lemma~\ref{lem:induce_id}.

\begin{lemma}
\label{lem:connectFn}
In Diagram~\eqref{eq:htpyCommutativeInduced},
assume that $P=Q$, $\mapP=\mapQ$ and $r=\mapP^n$ for some $n\geq 0$.
If $H\colon P \times [0,1]\to\Tor(\mapP)$ is a homotopy from $\incl_\mapP r$ to
$\incl_\mapP r \mapP$, i.e.,
$H(\cdot,0)=\incl_\mapP \mapP^n$ and $H(\cdot,1) = \incl_\mapP \mapP^{n+1}$,
then $(\mapP^n,H)_\iT \simeq \id_{\Tor(\mapP)}$.
\end{lemma}

\begin{proof}
A homotopy $H'\colon \Tor(\mapP) \times [0,1] \to \Tor(\mapP)$ is given by:
\[ H'([x,\theta],t) =
  \begin{cases}
	 \phi_\mapP([x,\theta],2t(n-\theta)) &\text{ for } 0 \leq t\leq 1/2, \\
	 H(x,(2t-1)\theta) &\text{ for } 1/2 \leq t \leq 1.
  \end{cases}
\]
Then $H'([x,\theta],0) = [x,\theta]$ and $H'([x,\theta],1) = H(x,\theta) = (\mapP^n,H)_\iT[x,\theta]$.
\end{proof}

Lemma~\ref{lem:induce_id} is the special case 
$H(x,\theta)=\phi_\mapP([\mapP^n(x),0],\theta)$
of Lemma~\ref{lem:connectFn}.

\begin{theorem}
\label{thm:shiftEquiHtpyEqui}
If the homotopy classes $[\mapP]$ and $[\mapQ]$ are shift equivalent,
then $\Tor(\mapP)$ and $\Tor(\mapQ)$ are homotopy equivalent.
\end{theorem}

\begin{proof}
This works similarly to the proof of Theorem~\ref{thm:torus_choice_index_pair}.
But here the induced maps depend on the chosen homotopies.
By assumption,
there are a homotopy $H\colon P\times[0,1]\to\Tor(\mapQ)$ such that
$H(x,0) = \incl_\mapQ r(x)$ and $H(x,1) = \incl_\mapQ r\mapP(x)$ for all $x\in P$,
and a homotopy
$H' \colon Q \times [0,1] \to \Tor(\mapP)$ such that
$H'(x,0) = \incl_\mapP s(x)$ and $H'(x,1) = \incl_\mapP s\mapQ(x)$ for all $x \in Q$.
We show that the composition $K=(s,H')_\iT \circ (r,H)_\iT$ is homotopic to the
identity on $\Tor(\mapP)$.

Since $[\mapP]$ and $[\mapQ]$ are shift equivalent, there is a homotopy
$L \colon P \times [0,1] \to P$ together with an $n\in\NN$ such that
$L(x,0) = sr(x)$ and $L(x,1)=\mapP^n(x)$ for
all $x\in P$.

Using a retraction from the square $[0,1] \times [0,1]$
to three of its boundary edges, there is a map
\[ F\colon P\times [0,1]\times[0,1] \to \Tor(\mapP) \]
such that for all $\theta\in [0,1]$ and $t \in [0,1]$:
\begin{enumerate}[label=(\roman*)]
\item $F(x,\theta,0) = K([x,\theta])$,
\item $F(x,0,t) = \incl_\mapP L(x,t)$, and
\item $F(x,1,t) = \incl_\mapP L(\mapP(x),t)$.
\end{enumerate}
Note that $F$ is well-defined, e.g.,
\[K([x,1]) = K([\mapP(x),0]) = [sr\mapP(x),0]=\incl_\mapP L (\mapP(x),0).\]
Additionally, $F(x,1,t)=F(\mapP(x),0,t)$. This means that
$F$ induces a continuous map $F'\colon \Tor(\mapP) \times [0,1] \to \Tor(\mapP)$.
Let
\begin{align*}
K' \colon P\times [0,1] &\to \Tor(\mapP),\\
(x,\theta) &\mapsto F'([x,\theta],1).
\end{align*}
Observe that $K'(x,0) = \incl_\mapP \mapP^n$ and $K'(x,1) = \incl_\mapP \mapP^{n+1}$
by construction. We get the following homotopies,
where the first one is given by $F'$ and
the second one by Lemma~\ref{lem:connectFn}:
\[ (s,H')_\iT \circ (r,H)_\iT = K \simeq (\mapP^n,K')_\iT \simeq \id_{\Tor(\mapP)}.\]
An analogous argument shows that $(r,H)_\iT\circ(s,H')_\iT \simeq \id_{\Tor(\mapQ)}$.
\end{proof}

The converse can easily be shown to be false.
For example, if $P=\{1\}$, $Q=\{1,2\}$ with the 
discrete topology and $\mapQ\colon Q\to Q, \mapQ(1)=2, \mapQ(2)=1$.
Then $\Tor(\mapP)=\Tor(\mapQ)=S^1$.
Suppose there is a  map $r\colon P \to Q$ such that $\mapQ r \simeq r \mapP$.
Then $\mapQ r(1)=r\mapP(1) = r(1)$, but $\mapQ(x) \neq x$ for all $x$. A contradiction.

The rest of this section deals with a specific situation in which the converse is true,
as described in Theorem~\ref{thm:htpyEquToShift}.
As a tool in the following proof, we use the \emph{mapping telescope}
(see also \citealt{H}, Sect. 3.F):
Let $P$ be a topological space and let $\mapP \colon P \to P$ be a continuous map.
Then let
\[ \Tel(\mapP) = \frac{\coprod_{i\in\ZZ}(P\times [0,1] \times \{i\})}
   {(x,1,i) \sim (\mapP(x),0,i+1)},\]
i.e., countably many mapping cylinders of $\mapP$ are glued together.
It is a covering space of $\Tor(\mapP)$ via
\[ \pi \colon \Tel(\mapP) \to \Tor(\mapP),\ (x,t,i) \mapsto (x,t) \text{ for }t\in [0,1).\]
For $n\in\NN$, let
\[\Tel_n(\mapP)=\left(\coprod_{i=-n}^{n-1} (P \times I \times \{i\}) \coprod P\times\{0\}\times\{n\}\right)
\mathlarger{\mathlarger{/}}{\sim}, \]
using the same identifications as in $\Tel(\mapP)$.
Then $\Tel_n(\mapP)$ deformation retracts to $P$ along a lift of the suspension
semiflow $\phi_\mapP$
via a map $\rho \colon \Tel_n(\mapP) \to P$.
Note that any compact subset of $\Tel(\mapP)$ is contained in $\Tel_n(\mapP)$ for
some $n\in\NN$. We first show two lemmas.

\begin{lemma}
\label{lem:htpy-exists}
Let $\mapP P\to P$ be a continuous map on a compact and connected topological space $P$,
and let $\alpha \colon P \to P$ be a continuous map such that $\alpha \mapP \simeq \mapP \alpha$.
Hence, there is a homotopy $H\colon P\times[0,1]\to\Tor(\mapP)$ such that
$H(\cdot,0) = \incl_\mapP \alpha$ and $H(\cdot,1) = \incl_\mapP \alpha\mapP$.
If $(\alpha,H)_\iT \simeq \id_{\Tor(\mapP)}$,
then there are $n,k\in\NN$ such that $\mapP^n \simeq\mapP^k \alpha$.
\end{lemma}

\begin{proof}
The assumptions yield that $(\alpha,H)_\iT \simeq \id_{\Tor(\mapP)}$, hence
\[\incl_\mapP = \id_{\Tor(\mapP)} \incl_\mapP \simeq (\alpha,H)_\iT \incl_\mapP  = \incl_\mapP \alpha,\]
i.e., there is
a homotopy $H' \colon P \times [0,1] \to \Tor(\mapP)$ such that $H'(\cdot,0) = \incl_\mapP$
and $H'(\cdot,1)=\incl_\mapP\alpha$.
The map $H'$ fits into the following commutative diagram
with solid arrows:
\[ \begin{tikzcd}
  P\times \{0\} \ar[hook]{r}{i_0} \ar[hook]{d} &\Tel(\mapP) \ar{d}{\pi} \\
  P\times [0,1] \ar{r}{H'} \ar[dashed]{ur}{\overline{H}}& \Tor(\mapP),
\end{tikzcd} \]
where $i_0(x,0) = (x,0,0).$
The dashed lift $\overline{H}$ exists and is the unique
map making the diagram commute
because $\pi$ is a covering projection
(\citealt{H}, Prop. 1.30).
Since the domain of $\overline{H}$ is compact, its image is a compact
subset of $\Tel(\mapP)$, hence there is an $n \in \NN$ such that
$\overline{H}=i\circ h$,
where $h\colon P\times [0,1] \to \Tel_n(\mapP)$ and
$i\colon\Tel_n(\mapP) \to \Tel(\mapP)$ is the inclusion.
Let $\rho\colon \Tel_n(\mapP) \to P$ be the deformation retraction
from following the flow lines of the suspension semiflow.
This yields a map
$\rho h \colon P\times [0,1] \to P$ with $\rho h(x,0) = \rho(x,0,0) = \mapP^n(x)$.
Since $\pi h(x,1) = H(x,1) = [\alpha(x),0]$, the image $h(P\times \{1\})$
is in the fiber $\coprod_{i=-n}^n P\times\{0\}\times\{i\}\subset\Tel_n(\mapP)$
over $P\times \{0\}\subset \Tor(\mapP)$.
Since $P$ is connected, 
$h(P\times \{1\}) \subset P\times \{0\} \times \{n-k\}$ for some $k \in \{0,\ldots,2n\}$.
Therefore $\rho h(x,1)=\mapP^k(\alpha(x))$,
and $\rho h$ is a homotopy $\mapP^n \simeq \mapP^k\alpha$.
\end{proof}

Now we show that the exponents in the definition of shift equivalence are allowed
to differ in the following sense.
\begin{lemma}
\label{lem:criterion-shift-equ}
Let $\mapP\colon P \to P$ and $\mapQ \colon Q \to Q$ be continuous maps, and assume
there are $r\colon P \to Q$ and $s\colon Q\to P$ such that
$r\mapP \simeq \mapQ r$, $\mapP s \simeq s\mapQ$.
Assume also that there are $n,m \in \NN$ such that
$sr \simeq \mapP^n$ and $rs \simeq \mapQ^m$. Then $[\mapP]$ and $[\mapQ]$ are shift equivalent.
\end{lemma}

\begin{proof}
The assumptions yield a homotopy commutative diagram
\[
  \begin{tikzcd}
	P \ar{dd}[swap]{\mapP^{n+m}} \ar{r}{sr}      & P \ar{r}{r} \ar{d}{\mapP^n} \ar{ddl}{sr}
	  & Q \ar{r}{rs} \ar{dl}{s} \ar{d}{\mapQ^m} & Q \ar{dd}{\mapQ^{n+m}} \ar{ddl}{rs} \\
	& P \ar{d}{\mapP^n} \ar{r}{r} & Q \ar[d,"\mapQ^m" near start] \ar{dl}{s} \\ 
	P \ar{r}{sr}      & P \ar{r}{r} & Q \ar{r}{rs} & Q.
  \end{tikzcd}
\]
We do not distinguish the upper and lower row (alternatively, one
can think of three copies of this diagram ``glued'' together).
There are several paths from the left $P$ to itself.
First going along the horizontal composition $rsrsr$ and 
then to the lower left, we get the homotopy
\[ (sr \mapP^n srs)(rsrsr) \simeq \mapP^{3(n+m)}. \]
In a similar manner, starting from the upper right $Q$, one sees
\[ (rsrsr)(sr \mapP^n srs) \simeq \mapQ^{3(n+m)}.\]
Hence, $[\mapP]$ and $[\mapQ]$ are shift equivalent.
\end{proof}

\begin{theorem}
\label{thm:htpyEquToShift}
Let $\mapP\colon P \to P$, $\mapQ\colon Q \to Q$, $r \colon P \to Q$
and $s \colon Q\to P$
such that $r\mapP\simeq \mapQ r$ and $s\mapQ \simeq \mapP s$.
Now suppose that $(sr)_\iT \simeq \id_{\Tor(\mapP)}$
and $(rs)_\iT \simeq \id_{\Tor(\mapQ)}$
and that $P$ and $Q$ are compact and connected. Then $[\mapP]$ and $[\mapQ]$
are shift equivalent.
\end{theorem}

\begin{proof}
Applying Lemma~\ref{lem:htpy-exists} to $\alpha=rs$ and $\alpha=sr$, respectively,
we see that there are $k,n,m \in \NN$ such that
$\mapP^n \simeq \mapP^k sr$ and $\mapQ^m \simeq \mapQ^k rs$. Hence the following diagram
commutes up to homotopy.
\[
  \begin{tikzcd}[row sep=large, column sep=large]
	P \ar{r}{\mapP^n} \ar{d}{r}      & P \ar{d}{r} \\
	Q \ar{r}{\mapQ^m} \ar{ur}{\mapP^k s} & Q
  \end{tikzcd}
\]
Now the result follows from Lemma~\ref{lem:criterion-shift-equ}.
\end{proof}

\section{Examples}
\label{sec:examples}

In order to discuss examples, we first recall the following results
about the homology and the fundamental group of mapping tori.

The following statement about homology is shown in \citet{H},
Example~2.48:
A continuous map $\mapP\colon P \to P$
induces a map $\mapP_*$ in homology and
this fits into a long exact sequence in homology.
\[ \begin{tikzcd}[]
\cdots \ar{r} & \Ho_n(P) \ar{r}{\id-\mapP_*} & \Ho_n(P)\ar{r}{{j_\mapP}_*}
  & \Ho_n(\Tor(\mapP)) \ar{r}{\partial} & \Ho_{n-1}(P) \ar{r} & \cdots
\end{tikzcd} \]

In the following examples, the spaces $P$ are finite wedges of circles
and $\mapP$ sends the base point to the base point.
Computing the fundamental group in this case works as follows.
Let $P=\bigvee_{i=1}^n S^1$. Each of the circles contributes a generator $a_i\in\pi_1(P)$
and $\pi_1(P) = \langle a_1,\ldots, a_n\rangle$.
In order to build $\Tor(\mapP)$, one adds a circle, say $z$, and for each
$i$ a $2$-cell attached to the $1$-skeleton along $a_i z \mapP_*(a_i)^{-1} z^{-1}$.
Hence, $\pi_1(\Tor(\mapP))$ is a quotient of the free group on $n+1$ generators
as follows (cf. \citealt{H}, Prop.~1.26):
\[ \pi_1(\Tor(\mapP))=\langle a_1,\ldots,a_n,z \mid a_1 z = z \mapP_*(a_1),
\ldots, a_n z = z \mapP_*(a_n)\rangle.\]
The space $\Torp(\mapP)$ for a pointed map $\mapP$ is constructed similarly,
but without adding the $1$-cell $z$. This
yields
\[ \pi_1(\Torp(\mapP))=\langle a_1,\ldots,a_n \mid a_1 = \mapP_*(a_1),\ldots,a_n=\mapP_*(a_n)\rangle.\]

\begin{example}
\label{ex:degree2}
Let $P=S^1$ and let $\mapP$ be a degree-$2$ map.
Then $\pi_1(S^1)\cong \langle a \rangle$
and the group $\pi_1(\Torp(f_{(N,L)}))$ is trivial
because $\mapP_*(a) = a^2$.
Therefore,
$\pi_1(\Torp(f_{(N,L)}))\cong\langle a \mid a= a^2\rangle = \{e\},$
the trivial group.
But $\pi_1(\Tor(f_{(N,L)})) \cong \langle a, z \mid az = z a^2\rangle$, which
is not the free group on one generator.
\end{example}

Since we would like to be able to compare Conley indices numerically,
we show how this can be done from the presentation of a group 
with finitely many generators as in Example~\ref{ex:degree2}. We
use the software package \citeauthor{GAP}.
Given $G$, the fundamental group from the example, the software
lists all the subgroups $S$ of $G$ with index $[G:S]\leq 3$
and then computes the abelianization for each of these $S$.
Excecuting the GAP code

\begin{verbatim}
F:=FreeGroup("a","z");
G:=F/ParseRelators(F,"az = za^2");
subgroups:=LowIndexSubgroupsFpGroup(G,3);
Print(List(subgroups,IndexInWholeGroup),"\n");   # indices
Print(List(subgroups,AbelianInvariants));  # abelianizations
\end{verbatim}
yields the output
\begin{verbatim}
[ 1, 2, 3, 3 ]
[ [ 0 ], [ 0, 3 ], [ 0, 7 ], [ 0 ] ]
\end{verbatim}
The groups are represented by giving torsion coefficients, e.g.,
$\mathtt{[ 0, 3 ]}$ represents the 
abelianization $\ZZ \times \ZZ/3\ZZ$ of a subgroup $S$ with index $2$.
In particular, the Conley index in our example is not trivial.

A commonly used algebraic invariant of the Conley index is the
\emph{homological Conley index}, by which we mean the shift equivalence class of 
reduced homology $\rHo(f_{(N,L)})$,
where linear maps $\mapP,\mapQ$ are shift equivalent if there are linear maps
$r,s$ such that $r\mapP = \mapQ s$, $s \mapQ = r\mapP$, $sr=\mapP^n$
and $rs=\mapQ^n$ for some $n\in\NN$.
This is not the only useful algebraic invariant as we show in the following example.

\begin{example}
The mapping torus index contains information
which the homological Conley index cannot represent.
Let $P=S^1\vee S^1$ with circles $a$ and $b$, and let $\mapP\colon P \to P$ 
such that $a \mapsto aba^{-1}b^{-1}$ and $b \mapsto a^{-1} bab^{-1}$,
This induces the trivial (zero) homomorphism in reduced homology.
The fundamental group of its mapping torus is
\[ \pi_1(\Tor(\mapP))= \langle a,b,z \mid az = z a b a^{-1}b^{-1},
  bz = z a^{-1} b a b^{-1} \rangle. \]
Similarly to the example above, GAP computes that $\pi_1(\Tor(\mapP))$ has a subgroup
with index $5$ and abelianization $\ZZ \times \ZZ/3\ZZ \times \ZZ/8\ZZ$.
The shift equivalence class of $[\mapP]$ is therefore not trivial.
But this information is not visible when using homology $\rHo_*(\mapP)$,
which is shift equivalent to the graded module homomorphism $0\to 0$.
It also seems hard to see that $[\mapP]$ is not shift equivalent
to the trivial map $\{*\} \to \{*\}$ directly from
the definition of shift equivalence.
\end{example}

\begin{example}
Let $f_1\colon \RR \to \RR, f_1(x) = 2x$. Then $\{0\}$ is an isolated fixed point
with index pair $([-2,2],[-2,-1]\cup[1,2])$, and the index map is homotopic to the
identity on $S^1$.
Its mapping torus index is $\CT(\{0\},f_1)=[S^1\times S^1]$,
and $\CTp(S,f_1)=[S^1\rtimes S^1]$.

Let $f_2\colon \RR\to\RR, f_2(x) = -x^3$. Then $\{-1,1\}$ is an isolated invariant set with
\[\CTp(\{-1,1\},f_2)=[S^1 \rtimes S^1] = \CTp(\{0\},f_1).\]
This equality can be seen using Theorem~\ref{thm:suspension_index}.
But the shift equivalence indices of $(\{0\},f_1)$ and $(\{-1,1\},f_2)$ differ.
Indeed, choosing some field $\FF$,
the homological Conley index of $(\{0\},f_1)$ in first homology
is the identity on $\FF$,
whereas the corresponding homological Conley index of $(\{-1,1\},f_2)$
is an automorphism of $\FF^2$. They cannot be shift equivalent
(\citealt{MM}).
\end{example}

\begin{example}
Let $g\colon \RR \to \RR$, $g(x)=-2x$. Then $\{0\}$ is an isolated fixed point
with index pair $([-2,2],[-2,-1]\cup[1,2]\})$ and $g_{(N,L)}$ is a map on $S^1$
with degree $-1$. Hence the mapping torus index of $(\{0\},g)$ is 
(the homotopy type of) the Klein bottle.
\end{example}

\begin{example}
We recall Example~6.1 from \citet{Sz}. This example
has non-trivial shift equivalence index,
whereas the indices defined 
by \citet{M90} and \citet{RS} are trivial.
The mapping torus index offers some more insight:
Let
$X=(-\infty,0]\cup \{2^{-k} \mid k \in \NN\} \subset \RR$
and $h\colon X \to X,$
\[ h(x) = \begin{cases} 2x &\text{if } x \leq 0, \\
                        x/2 &\text{if } x \geq 0.
		  \end{cases} \]
Then $(N,L) = (X\cap[-2,1],[-2,-1])$ is an index pair for $(\{0\},h)$.
Its mapping torus index $\CT(\{0\},h)$
is the union of a helix approaching $\{0\}\times S^1$ and the
cylinder $[-1,0]\times S^1$.
This space is compact and connected,
but neither pathwise connected nor locally connected.
In particular, $\CT(\{0\},h) \neq [S^1]$.
\end{example}

\begin{example}[Smale's horseshoe]
\label{ex:horseshoe}
Consider Smale's $U$-horseshoe $f\colon \RR^2 \to \RR^2$,
a homeomorphism which bends the unit square $N=[0,1]\times[0,1]$ to form a horseshoe,
which is sketched in Fig.~\ref{fig:horseshoe}.
We consider the following index pair
used in \citet{M90}, Example~8.1, and \citet{MM}.
Let $L_1 = [0,1]\times[0,1/5]$, $L_2 = [0,1]\times[2/5,3/5],$
$L_3 = [0,1]\times [4/5,1]$ and $L=L_1\cup L_2\cup L_3$.
Then $(N,L)$ is an index pair for $(\Inv(N,f),f)$ and
$N/L \simeq S^1 \vee S^1$ with fundamental group the free group
on two generators $a$ and $b$, where
$f_{(N,L)*}(a)=ab$ and $f_{(N,L)*}(b)=b^{-1}a^{-1}$. Now
\[ \pi_1(\Tor(f_{(N,L)}))\cong \langle a,b,z \mid
az=zab,bz=zb^{-1}a^{-1} \rangle \cong \langle z \rangle \cong \ZZ.\]
This can be seen by observing that the left relation is equivalent
to $b^{-1}a^{-1}=z^{-1}a^{-1}z$ and inserting this into the right relation.

\colorlet{lightGray}{gray!50}
\tikzset{
  pics/carc/.style args={#1:#2:#3}{
    code={
      \draw[pic actions] (#1:#3) arc(#1:#2:#3);
    }
  }
}
\begin{figure*}
\centering
\begin{subfigure}[t]{0.3\textwidth}
\raisebox{1.15cm}{
\begin{tikzpicture}[scale=2.5,path/.append style={
   postaction=decorate,
   thick,
 }]
  \foreach \y in {0,0.4,0.8} 
    \fill[lightGray] (0,\y) rectangle (1,\y+0.2);
   \draw[path] (0,0) |- (1,1);
   \draw[path] (0,0) -| (1,1);
   \draw[path] (0,0.2) -- (1,0.2);
   \node at (0.5,0.1) {\small $L_1$};
   \draw[path] (0,0.4) -- (1,0.4);
   \node at (0.5,0.5) {\small $L_2$};
   \draw[path] (0,0.6) -- (1,0.6);
   \node at (0.5,0.9) {\small $L_3$};
   \draw[path] (0,0.8) -- (1,0.8);
   \draw[thin,->] (0.3,0.1) -- (0.3,0.3);
   \draw[thin] (0.3,0.3) -- (0.3,0.5);
   \node at (0.23,0.28) {\small $a$};
   \draw[thin,->] (0.7,0.5) -- (0.7,0.7);
   \draw[thin] (0.7,0.7) -- (0.7,0.9);
   \node at (0.77,0.69) {\small $b$};
\end{tikzpicture}
}
\end{subfigure}
\begin{subfigure}[t]{0.3\textwidth}
\raisebox{.1cm}{
\begin{tikzpicture}[scale=2.5,path/.append style={
   postaction=decorate,
   thick,
 }]
   \fill[lightGray] (0.2, 1) arc [start angle=180, delta angle=-180, radius=0.3];
   \fill[white] (0.4, 1) arc [start angle=180, delta angle=-180, radius=0.1];
   \fill[lightGray] (0.2,0) rectangle (0.4,-0.2);
   \fill[lightGray] (0.6,0) rectangle (0.8,-0.2);
   \draw[path] (0,0) -- (0,1);
   \draw[path] (0,1) -- (1,1);
   \draw[path] (1,1) -- (1,0);
   \draw[path] (1,0) -- (0,0);
   \draw[path] (0.2,-0.2) -- (0.2,1);

   \draw[path] (0.2, 1) arc [start angle=180, delta angle=-180, radius=0.3];
   \draw[path] (0.8,1) -- (0.8,-0.2);
   \draw[path] (0.8,-0.2) -- (0.6,-0.2);
   \draw[path] (0.6,-0.2) -- (0.6,1);
   \draw[path] (0.4, 1) arc [start angle=180, delta angle=-180, radius=0.1];
   \draw[path] (0.4,1) -- (0.4,-0.2);
   \draw[path] (0.4,-0.2) -- (0.2,-0.2);
   \node at (0.3,-0.3) {\small $f(L_1)$};
   \node at (0.5,1.18) {\small $f(L_2)$};
   \node at (0.7,-0.3) {\small $f(L_3)$};
\end{tikzpicture}
}
\end{subfigure}
\begin{subfigure}[t]{0.3\textwidth}
\begin{tikzpicture}[scale=2.5,path/.append style={
   postaction=decorate,
   thick,
 }]
   \fill[lightGray] (0.2,0) rectangle (0.4,-0.2); 
   \fill[lightGray] (0.6,1) rectangle (0.8,1.2); 

   \draw[lightGray, line width=0.5cm] (0.8,1) pic{carc=0:180:1.25};
   \fill[lightGray] (1.2,1.02) rectangle (1.4,-0.02);
   \draw[lightGray, line width=0.5cm] (1,0) pic{carc=180:0:-0.75};
   \draw[path] (0,0) -- (0,1);
   \draw[path] (0,1) -- (1,1);
   \draw[path] (1,1) -- (1,0);
   \draw[path] (1,0) -- (0,0);
   \draw[path] (0.2,-0.2) -- (0.2,1);
   \draw[path] (0.2, 1) arc [start angle=180, delta angle=-180, radius=0.6];
   \draw[path] (1.4,1) -- (1.4,0);
   \draw[path] (1.4,0) arc [start angle=180, delta angle=-180, radius=-0.4];
   \draw[path] (0.6,0) -- (0.6,1.2);
   \draw[path] (0.6,1.2) -- (0.8,1.2);
   \draw[path] (0.8,1.2) -- (0.8,0);
   \draw[path] (1.2,0) arc [start angle=180, delta angle=-180, radius=-0.2];
   \draw[path] (1.2,0) -- (1.2,1);
   \draw[path] (0.4, 1) arc [start angle=180, delta angle=-180, radius=0.4];
   \draw[path] (0.4,1) -- (0.4,-0.2);
   \draw[path] (0.4,-0.2) -- (0.2,-0.2);
   \node at (0.3,-0.3) {\small $g(L_1)$};
   \node at (0.22,1.5) {\small $g(L_2)$};
   \node at (0.98,1.1) {\small $g(L_3)$};
\end{tikzpicture}
\end{subfigure}
\caption{The components of $L$, the exit set for Example~\ref{ex:horseshoe}, and their images
under $f$ and $g$, respectively.}
\label{fig:horseshoe}
\end{figure*}
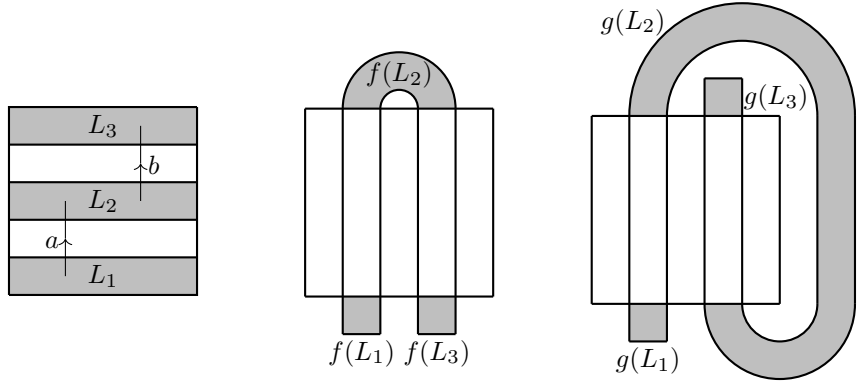

Similarly $(N,L)$ is an index pair for $(\Inv(N,g),g)$, where $g\colon \RR^2 \to \RR^2$ is
the $G$-horseshoe. The images of $N$ and $L$ are sketched in the right subfigure of Fig.~\ref{fig:horseshoe}. 
Here, $g_{(N,L)*}(a) = ab$ and $g_{(N,L)*}(b) = ab$, hence
\[ \pi_1(\Tor(g_{(N,L)}))\cong \langle a,b,z\mid az=zab,bz=zab \rangle
  \cong \langle a, z \mid az=za^2 \rangle, \]
the group from Example~\ref{ex:degree2}. Hence, its (unreduced) mapping torus index is non-trivial.
On the contrary, $\pi_1(\Torp(g_{(N,L)}))$ is trivial.
\end{example}

\section{Definition via suspension semiflow on $X$}
\label{sec:compareFloer}

The reduced mapping torus index is equivalent to the flow Conley index
of the suspension semiflow in the sense presented here.
The results of this section are also included in
Sect.~4 of the unpublished paper
\emph{Morse inequalities and zeta functions}\footnote{
currently accessible at
\url{https://people.math.ethz.ch/~salamon/PREPRINTS/zeta.pdf}}
written by J.~W. Robbin, D.~A. Salamon and E.~C. Zeeman in 1989.
The idea of considering the Conley index for the suspension semiflow
also appeared in~\citet{F}.

Let $f\colon X \to X$ be a discrete dynamical system. For a set $A\subset X$,
let $\suspend A := q(A\times [0,1]) \subset \Tor(f)$,
where $q\colon X\times[0,1]\to\Tor(f)$ is the quotient map; i.e., 
\[\suspend A = \{ [x,\theta] \in \Tor(f) \mid x\in A, \theta\in[0,1]\}
=\frac{A\times[0,1]\cup f(A)\times\{0\}}{(x,1)\sim(f(x),0)}.\]
Note that, given an invariant set $S$ of $f$, $\suspend S = \Tor(f\!\!\restriction_S)$,
the mapping torus of the restriction of $f$ to $S$.

For the following proof of Theorem~\ref{thm:suspension_index}, we
recall a special kind of index pair for maps.
A compact pair $(N,L)$ is a \emph{strong index pair} for an isolated
invariant set $S$ of $f$ if
\begin{enumerate}
  \item $S=\Inv(\cl(N\setminus L),f) \subset \Int (N\setminus L)$,
  \item\label{defStrongIndexPair2} $f(L) \cap N \subset L$, and
  \item\label{defStrongIndexPair3} $f(N\setminus L) \subset N$.
\end{enumerate}

A strong index pair exists for every isolated invariant set $S$
(\citealt{Sz}, Theorem~3.1;
\citealt{MM}, Theorem~3.25)
and a strong index pair is an index pair
(\citealt{RS}, Corollary~4.4).

We recall the definition of an index pair for a semiflow
$\phi\colon X\times [0,\infty) \to X$ given by \citet{C}.
Given a subset $M \subset X$, its invariant part is
\[ \Inv(M,\varphi) = \{ x \in M \mid \text{ there is a solution 
$\gamma\colon\RR\to M$ of $\varphi$ with $\gamma(0)=x$}\}. \]
A compact pair $(\tilde{N},\tilde{L})$ is an \emph{index pair for} 
$(S,\phi)$ if the following conditions are fulfilled:

\begin{enumerate}
\item\label{defIndexPairFlow1} $\Inv(\cl(\tilde{N}\setminus \tilde{L}),\varphi) \subset \Int(\tilde{N}\setminus \tilde{L})$;
\item\label{defIndexPairFlow2} if $x\in \tilde{L}, t>0, \varphi(x,[0,t]) \subset \tilde{N}$,
  then $\varphi(x,[0,t]) \subset \tilde{L}$;
\item\label{defIndexPairFlow3} if $x \in \tilde{N}, t>0, \varphi(x,t) \notin \tilde{N}$,
  then there is a $t' \in [0,t]$ such that $\varphi(x,t') \in \tilde{L}$ and $\varphi(x,[0,t'])\subset \tilde{N}$.
\end{enumerate}

\begin{lemma}
\label{lem:indexPairs}
Let $(N,L)$ be a strong index pair for $(S,f)$. Then $(\suspend N,\suspend L)$ is
an index pair for $(\suspend S,\phi_f)$.
\end{lemma}

\begin{proof}
For~\ref{defIndexPairFlow1}, we
first show that $\suspend S\subset\Int(\suspend N\setminus \suspend L)$:
Let $x\in S, 0\leq\theta\leq 1$. It suffices to consider $\theta<1$
(since $[x,1]=[f(x),0]$).
Pick an open neighborhood $U$ of $x$ with $U\subset N\setminus L$.
If $\theta\in(0,1)$, one gets $[x,\theta]\in q(U\times(0,1))\subset \suspend N\setminus \suspend L$.
If $\theta=0$, let $V=f^{-1}(U)\subset N\setminus L$.
Then $q(U\times [0,1) \cup V\times (0,1]) \subset \suspend N\setminus \suspend L$
is an open neighborhood of $[x,0]$.

Now we show $\suspend S=\Inv(\cl(\suspend N\setminus \suspend L),\phi_f)$:
The inclusion ``$\subset$'' holds because $\suspend S$ is obviously
an invariant set for $\varphi_f$. For the other inclusion, first observe that
\[ \suspend N\setminus \suspend L \subset \suspend (N\setminus L) \subset \suspend  \cl (N\setminus L).\]
The left inclusion follows because $[x,\theta]\in \suspend N\setminus \suspend L$
implies $x\in N\setminus L$. These sets differ in general
because if $x\in N\setminus L$ and $f(x)\in L$,
then $[x,1]=[f(x),0]\in \suspend L$.
Since the set on the right is compact,
we get $\cl (\suspend N\setminus \suspend L) \subset \suspend \cl(N\setminus L)$.
Now let $[x,\theta] \in \Inv(\cl(\suspend N\setminus \suspend L),\phi_f)$ with $0\leq \theta\leq 1$,
i.e., there is a solution $\gamma\colon \RR \to \cl(\suspend N\setminus \suspend L)$ of $\phi_f$
with $\gamma(0) = [x,0] \in \cl(\suspend N\setminus \suspend L)$.
The curve $\gamma$ yields a solution $\bar\gamma\colon \ZZ \to \cl(\suspend N\setminus \suspend L)$ 
of $f$ with $\bar\gamma(0) = x$ and
$[\bar\gamma(n),0] \in \cl(\suspend N\setminus \suspend L)\subset \suspend \cl (N\setminus L)$ for
all $n\in\ZZ$. Therefore, $\bar\gamma(n)\in\cl(N\setminus L)$.
Since $\cl(N\setminus L)$ is an isolating neighborhood for $S$,
this yields $x\in S$ and hence $[x,\theta]\in \suspend S$.

For~\ref{defIndexPairFlow2}, consider
$x\in L$, $0\leq\theta\leq 1$,
and a trajectory $\phi_f([x,\theta],[0,t]) \subset \suspend N$ for some $t>0$.
Then $\phi_f([x,0],[0,\theta+t]) \subset \suspend N$.
Let $n:= \lfloor \theta + t \rfloor.$
Then $f^k(x) \in L$ for each $0\leq k \leq n$
since $f(L) \cap N \subset L$.
Therefore, $\phi_f([x,\theta],[0,t]) \subset \suspend L$.

For~\ref{defIndexPairFlow3}, consider some $x\in N$ with $0\leq \theta \leq 1$,
and suppose that $\phi_f([x,\theta],t)=\phi_f([x,0],\theta+t)\notin \suspend N$
for some $t>0$.
Let $n:=\lfloor \theta+t \rfloor$.
Then $f^n(x) \notin N$. Let
\[ m:=\max \{ k\in\mathbb{N} \mid f^i(x) \in N \text{ for all } 0\leq i \leq k\} < n. \]
Then $\phi_f([x,\theta],[0,m]) \subset \suspend N$.
Now assume $f^m(x) \notin L$. Then $f^{m+1}(x) \in N$ because
$f(N\setminus L) \subset N$. This contradicts the definition of $m$.
Overall, $f^m(x) \in L$, and therefore $\varphi_f([x,\theta],m) \in \suspend L$.
\end{proof}

\begin{lemma}
\label{lem:quotientTori}
If $(N,L)$ is a strong index pair for $(S,f)$, then $\Torp(f_{(N,L)}) = \suspend N/\suspend L.$
\end{lemma}

\begin{proof}
Using $f(N\setminus L) \subset N$, we have
\[ \suspend N = \frac{N\times [0,1] \cup f(L)\times \{0\}}
{(x,1)\sim(f(x),0)\text{ for }x\in N}\subset\Tor(f). \]
Taking the quotient yields
\[ \suspend N/\suspend L 
= \frac{N\times[0,1]/L\times[0,1]}{(x,1)\sim(f(x),0)\text{ for } x,f(x)\in N}
 = \Torp(f_{(N,L)}).\]
\end{proof}

Since we are free to choose an arbitrary index pair for $(S,f)$ and $(\suspend S,\phi_f)$,
respectively,
Lemmas~\ref{lem:indexPairs} and \ref{lem:quotientTori} yield
\begin{theorem}
\label{thm:suspension_index}
For every isolated invariant set $S$ of $f$,
$\CTp(S,f)$ is the flow Conley index of $(\suspend S,\phi_f)$.
\end{theorem}

\section{Numerical representation}
\label{sec:numerics}

In this section, we sketch an idea for the numerical representation
of the mapping torus of some self-map $f\colon X\to X$.
Using interval arithmetic in rigorous numerics
(see \citealt{KMM}; or \citealt{BGHKMOP}), one can construct
a numerical representation of a covering $Z\subset X\times X$ of the graph
$G(f)$ of $f$.
The sets $Z$ and hence the maps $p$ and $q$ can be represented on a computer,
even though $f$ is not directly known.
Letting $p(x,y)=x$ and $q(x,y)=y$, the map $f$ factors through
$G(f)$ and hence through $Z$ as follows:
Let 
\begin{align*}
\tilde{f}\colon X&\to Z,\\
 x&\mapsto (x,f(x)).
\end{align*}
Then
$f = q \circ \tilde{f}.$
For the diagram $p,q\colon Z\rightrightarrows X$,
we consider its homotopy colimit,
which we also call the \emph{mapping torus of $p$ and $q$} here
(cf. \citealt{H}, Example~2.48),
\begin{equation*}
\label{doubleMappingTorus}
\Tor(p,q):=\frac{(Z \times [0,1]) \coprod X}{(z,0)\sim p(z), (z,1)\sim q(z)}.
\end{equation*}
Analogously, we define $\Tor(\id_X,f)$,
the mapping torus of $\id_X$ and $f$, as the quotient of $(X \times [0,1]) \coprod X$.
It is is homotopy equivalent to $\Tor(f)$ as defined in Sect.~\ref{sec:defIndex}.
Now let 
\[(\tilde{f},\id)_\iT\colon \Tor(\id_X,f)\to\Tor(p,q)\]
be the map induced on the summands by
\[ \tilde{f}\times \id_{[0,1]}\colon X \times [0,1] \to Z\times[0,1] \]
and the identity on $X$, respectively.
Now, $\Tor(p,q)$ is potentially useful for representing $\Tor(f)\simeq\Tor(\id,f)$ because of the following
property.
\begin{proposition}
\label{prop:smale}
If $\tilde{f}$ is a homotopy equivalence, then 
$(\tilde{f},\id_X)_\iT$ is a homotopy equivalence.
\end{proposition}
\begin{proof}
This follows from the main property of homotopy colimits
(\citealt{Kozlov}, Lemma~15.12; \citealt{H}, Prop. 4G.1).
\end{proof}
From the representations in rigorous numerics,
it seems hard to show that $\tilde{f}$ is a homotopy equivalence.
But the algorithms therein can construct
an enclosure $Z$ of the graph of $f$ such that
$p^{-1}(x)$ has the homology of the one-point space for all $x\in X$.
In this case, the Vietoris mapping theorem
(\citealt{Vietoris})
shows that $\tilde{f}$ induces an isomorphism in homology
(similar theorems for homotopy groups exist, cf. \citet{Smale}).
Then the following proposition
offers a way to compute the homology of $\Tor(f)$.
\begin{proposition}
\label{prop:vietoris}
If $\tilde{f}$ induces an isomorphism in homology, then
$(\tilde{f},\id_X)_\iT$ induces an isomorphism in homology
$\Ho_*(\Tor(f))\cong\Ho_*(\Tor(p,q))$.
\end{proposition}
\begin{proof}
Let $j(x)=[x,0]$ be the inclusion of $X$ into the
respective mapping torus.
Example~2.48 from \citet{H} shows the construction of
homomorphisms $\partial$ such that the
upper and lower sequence in the following diagram are exact:
\[ \begin{tikzcd}
\cdots \ar{r} & \Ho_n(X) \ar{r}{\id-f_*} \ar{d}{\tilde{f}_*}[swap]{\cong} &
    \Ho_n(X)\ar{r}{\incl_*} \ar{d}{\id} &
	\Ho_n(\Tor(\id_X,f)) \ar{r}{\partial} \ar{d}{((\tilde{f},\id_X)_\iT)_*} & 
	\Ho_{n-1}(X)\ar{r} \ar{d}{\tilde{f}_*}[swap]{\cong}  & \cdots \\
\cdots \ar{r} & \Ho_n(Z) \ar{r}{p_*-q_*} & \Ho_n(X)\ar{r}{\incl_*}
	& \Ho_n(\Tor(p,q)) \ar{r}{\partial} & \Ho_{n-1}(Z)\ar{r}& \cdots
\end{tikzcd} \]
The left and middle square commute because
the underlying squares of continuous maps commute.
In order to apply the $5$-lemma,
it remains to check that the right square commutes.
By construction,
\[\partial\colon \Ho_n(\Tor(p,q)) \to \{(\alpha,-\alpha)\mid\alpha\in\Ho_{n-1}(Z)\} \cong \Ho_{n-1}(Z)\]
is the composition of
homomorphisms which are either
part of the long exact sequence of some pair
of spaces or induced by a continuous map.
The desired commutativity follows from the naturality of
these long exact sequences in homology.
\end{proof}

Note that the numerical computation of the homology groups of
some space (here, $\Tor(p,q)$) requires less machinery
than computing the induced map in homology $\Ho_*(f)$.
This might be useful in situations where it is not obvious how to
extract $\Ho_*(f)$ from the representation of $f$, for example
when only a noisy sample of pairs $(x,f(x))$ is given as in \citet{EJM}.
But also if one has an enclosure $Z$ of the graph as above, one could
avoid computing the induced homology $\Ho_*(f)$ if
the homology $\Ho_*(\Tor(f))$ of the mapping torus already contains
the relevant information.

\textbf{Acknowledgements.}
The author would like to thank Marian Mrozek for numerous fruitful discussions
during the development of this article and Dietmar Salamon for
suggesting the unpublished preprint that was helpful
for writing Sect.~\ref{sec:compareFloer}.

\end{document}